\documentclass{amsart}
\usepackage{amsrefs,mathrsfs,math}
\usepackage{lineno,tikz,tikz-cd}
\def\pair<#1>{{\langle\!\langle}#1{\rangle\!\rangle}}
\newcommand\src{{\operatorname{src}}}
\newcommand\img{{\operatorname{im}}}
\newcommand\GL{{\mathsf{GL}}}
\newcommand\PGL{{\mathsf{PGL}}}

\newtheorem{mainthm}{Theorem}
\renewcommand{\themainthm}{\Alph{mainthm}}

\begin{document}
\title{Representation zeta functions of self-similar branched groups}
\author{Laurent Bartholdi}
\date{14 April 2015}
\address{\'Ecole Normale Sup\'erieure, Paris, France\textit{\and} Georg-August Universit\"at zu G\"ottingen, G\"ottingen, Germany}
\email{laurent.bartholdi@gmail.com}
\thanks{The author is supported by DFG
  research grants BA4197/x and ANR ``@raction'' grant ANR-14-ACHN-0018-01. Part of the research was done at the
  Mittag-Leffler Institute, Stockholm and the University of Chicago, whom I thank for their hospitality.}

\subjclass{\textbf{20C15} (Ordinary representations and characters),
  \textbf{11M41} (Other Dirichlet series and zeta functions),
  \textbf{20E22} (Extensions, wreath products, and other
  compositions), \textbf{20E08} (Groups acting on trees),
  \textbf{20F65} (Geometric group theory)}
\begin{abstract}
  We compute the numbers of irreducible linear representations of
  self-similar branched groups, by expressing these numbers as the
  co\"efficients $r_n$ of a Dirichlet series $\sum r_n n^{-s}$.

  We show that this Dirichlet series has a positive abscissa of
  convergence and satisfies a functional equation thanks to which it
  can be analytically continued (through root singularities) to the
  left half-plane.

  We compute the abscissa of convergence and the functional equation
  for some prominent examples of branched groups, such as the Grigorchuk
  and Gupta-Sidki groups.
\end{abstract}
\maketitle

\section{Introduction}
Let $G$ be a group, and let $\widehat G$ denote its set of equivalence
classes of irreducible, finite-dimensional complex linear
representations; assume that there are finitely many such
representations in each degree ($G$ is then called \emph{rigid}). The
\emph{representation zeta function} of $G$ is the Dirichlet series
with integer co\"efficients
\[\zeta_G(s)=\sum_{\rho\in\widehat G}(\deg\rho)^{-s}=\sum_{n\ge1} r_n n^{-s},\]
with $r_n$ denoting the degree-$n$ representations in $\widehat G$,
that is, irreducible representations of $G$ in $\GL_n(\C)$. If the
numbers $r_n$ grow polynomially, then analytic properties of $\zeta_G$
yield asymptotic information on $r_n$ and conversely. For example, let
$\sigma_0(G)$ denote the abscissa of convergence of $\zeta_G$; then,
assuming $\sum r_n=\infty$,
\begin{equation}\label{eq:abscissa}
  \sigma_0(G)=\limsup_{n\to\infty}\frac{\log\sum_{j=1}^n r_j}{\log n},
\end{equation}
so the partial sums $\sum_{j=1}^n r_j$ grow approximately as
$n^{\sigma_0}$.  More precisely, the Landau-Phragm\'en theorem implies
that $\zeta_G(s)$ has a singularity at $\sigma_0$, and if a limiting
behaviour $\zeta_G(s)=(s-\sigma_0)^e g(s)+h(s)$ is known with $g,h$
holomorphic in $\{\Re(s)\ge\sigma_0\}$ and $e\in\R\setminus\N$, then
\[\sum_{j=1}^n r_j\approx\frac{g(\sigma_0)}{\sigma_0\Gamma(-e)}n^{\sigma_0}(\log n)^{-e-1},\]
see~\cite{tenenbaum:apnt}*{Theorem~15, page~243} and~\cite{dress:oscillation}.

Note that it is easy to deduce the number of linear representations of
given degree out of the number of irreducible ones, and vice versa;
indeed every linear representation decomposes into a direct sum of
irreducibles whose multiplicities are uniquely determined. Letting
$R_n$ denote the number of representations of degree $n$, we have the
Euler-product formula
\[\sum_{n\ge0} R_n t^n = \prod_{n\ge1}\left(\frac1{1-t^n}\right)^{r_n}.\]

\subsection{Self-similar branched groups}
Representation zeta functions have been extensively investigated for
linear groups (see~\S\ref{ss:intro:linear} for a quick summary); in
this article, we focus on \emph{self-similar branched groups}. They
are certain kinds of groups $G$ equipped with an injective
homomorphism $\psi\colon G\to G^d\rtimes\sym d$, and possessing a
finite-index subgroup $K$ such that $\psi(K)$ contains $K^d$; see
Definitions~\ref{def:ssg}, \ref{def:essg} and~\ref{def:branch} for the
exact definitions.

Thus in particular $G$ and $G^d$ have isomorphic finite-index
subgroups. The integer $d>1$ is called the \emph{degree} of the
branched group, and one says that $G$ is branched \emph{over}
$K$. Iterating the map $\psi$ on its components, one obtains for every
branched group an action by permutation on the set $X^*$ of words over
an alphabet $X$ of cardinality $d$.

Self-similar branched groups constitute a well-studied class of
groups, containing such prominent examples as Grigorchuk's torsion
group of intermediate word growth~\cite{grigorchuk:growth} and
Gupta-Sidki's examples~\cite{gupta-s:burnside}. Their topological
closures in $\aut(X^*)$ may be thought of as analogues of algebraic
groups, defined by equations over infinitely many variables indexed by
$X^*$, see~\cite{siegenthaler:phd}.

On the one hand, branched groups have many finite quotients coming
from the action of the group on $X^n$ for all $n\in\N$; so have many
finite, and in particular linear, representations. On the other hand,
they contain abelian subgroups of arbitrarily large rank and large
normalizer, so they are quite different from linear groups.

In this article, I show that the zeta function of a self-similar
branched group admits quite remarkable properties:
\begin{mainthm}\label{thm:main}
  Let $G$ be a self-similar group of degree $d>1$, branched over its
  subgroup $K$. Then $G$ is rigid if and only if $K/[K,K]$ is
  finite. In that case, its representation zeta function $\zeta_G$
  \begin{enumerate}
  \item has a positive, finite abscissa of convergence $\sigma_0$, so
    that the co\"efficients $r_n$ grow polynomially;
  \item is a linear combination of the solutions $\zeta_i(s)$ of a system of
    functional equations of the form
    \begin{equation}\label{thm:main:fe0}
      \begin{cases}
        \begin{array}[b]{r@{}l}F_i( &\zeta_1(s),\zeta_1(2s),\dots,\zeta_1(ds),\\
          &\zeta_2(s),\dotfill,\zeta_2(ds),\\
          &\dotfill\\
          &\zeta_N(s),\dotfill,\zeta_N(ds)
        \end{array}\kern-0.5em)=\zeta_i(s),\qquad i=1,\dots,N
      \end{cases}
    \end{equation}
    for some $N,P\in\N$ and some Dirichlet polynomials
    $F_1,\dots,F_N\in\Q[z_{1,1},\dots,z_{N,d},2^{-s},\dots,P^{-s}]$;
    furthermore, if $z_{j,k}$ have degree $k$, then the polynomials
    $F_i$ are homogeneous of degree $d$;
  \item can be continued to a bounded, multivalued analytic function
    on the half-plane $\Re(s)>0$, with only root singularities;
  \item has a Puiseux series expansion at $\sigma_0$ of the form
    \[\zeta_G(s)=\sum_{n=0}^\infty a_n (s-\sigma_0)^{n/e},\]
    for some integer $e\le d$.
  \end{enumerate}
\end{mainthm}

The functional equation can be determined algorithmically out of the
description of $G$ as a self-similar branched group, and has been
implemented in \textsc{Gap} code; it is part of the author's
package~\textsc{Fr} freely available on Internet. This code was used
to compute the various examples in~\S\ref{ss:illustrations}.

Note that in general the functional equation~\eqref{thm:main:fe0} is
not sufficient to determine $\zeta_G$. However, under a judicious
choice of finite data extracted from $G$, it determines $\zeta_G$ and
permits a very efficient calculation of its co\"efficients.

It is easy to generalize Theorem~\ref{thm:main} to more general
character series. Let us say that an element $g\in G$ of a
self-similar group is \emph{finite-state} if there exists a finite
subset $W\subseteq G$, containing $g$, such that $\psi(W)\subseteq
W^d\times\sym d$. In words, the element $g$ is defined by a finite set of
recursive rules via the map $\psi$. For any $g\in G$, write its Lambda
series
\[\Lambda(g,s)=\sum_{\rho\in\widehat G}\tr\rho(g)\,\deg(\rho)^{-s}.\]
In particular, $\Lambda(1,s)=\zeta_G(s-1)$. The proof of
Theorem~\ref{thm:main} actually gives:
\def\themainthm{\ref{thm:main}'}
\begin{mainthm}
  Let $G$ be a self-similar branched group of degree $d>1$, and let
  $g\in G$ be finite-state. Then all properties of $\zeta_G$ claimed
  in Theorem~\ref{thm:main} also hold for $\Lambda(g,s)$.
\end{mainthm}
In particular, the variables in the functional equation shall be of
the form $\Lambda_i(w,s)$ for all $w\in W$ and $i=1,\dots,N$, and the
co\"efficients in the functional equation will belong to the field
generated by the character values of $G$. I omit details.

Theorem~\ref{thm:main} extends the main results
of~\cite{bartholdi-h:zetawreath}, in which the group $G$ was assumed
to be isomorphic to $G\wr_X Q$. Here and below the \emph{wreath
  product} $G\wr_X Q$ of the group $G$ with the group $Q$, along the
$Q$-set $X$, is by definition $G^X\rtimes Q$, and \textbf{\color{red}
  we} write $G\wr Q$ if $X=Q$ with its regular $Q$-action.  I will
make liberal use of results from~\cite{bartholdi-h:zetawreath}.

\subsection{Historical background}\label{ss:intro:linear}
If the group $G$ is a topological or algebraic group, then it is
natural to restrict to continuous, respectively rational
representations. Since these behave usually much better, part of the
art is to relate the representation zeta function of a topological
(e.g. Lie) group to that of its lattices.

It seems that the first occurrence of representation zeta functions is
in~\cite{witten:quantumgauge}, which relates $\zeta_G(2g-2)$ to the
moduli space of flat connections of $G$-principal bundles over
$\Sigma_g$, for $G$ a compact, simple, simply connected Lie group and
$\Sigma_g$ an orientable surface of genus $g\ge2$. However, $\zeta_G$
was already implicitly considered earlier; for example, it follows
from Weyl's theory that, if $\ell$ be $G$'s rank and $\kappa$ be the
number of positive roots of $G$'s Lie algebra over $\C$, then there
exists a polynomial $P$ of degree $\kappa$ in $\ell$ variables such
that
\[\zeta_G(s)=\sum_{n_1\ge0,\dots,n_\ell\ge0}P(n_1,\dots,n_\ell)^{-s}.\]
It follows that the abscissa of convergence of $\zeta_G$ is
$\ell/\kappa$, and that $\zeta_G$ extends to a meromorphic function on
the whole
plane; see~\cite{larsen-lubotzky:representationlinear}*{Theorem~5.1}.

Larsen and Lubotzky consider
in~\cite{larsen-lubotzky:representationlinear} arithmetic lattices in
semisimple algebraic groups $\mathbf G$, and show, under the
``congruence subgroup property'', that these lattices $\Gamma=\mathbf
G(\mathcal O)$ are products of local factors $\mathbf G(\mathcal O_v)$
and archimedian factors $\mathbf G(\C)$; consequently, the
representation zeta function $\zeta_\Gamma$ is the product of the
respective zeta functions; for example,
\[\zeta_{\mathbf {SL_3}(\Z)}(s)=\zeta_{\mathbf {SL_3}(\C)}(s)\prod_{p\text{ prime}}\zeta_{\mathbf {SL_3}(\Z_p)}(s).\]
A careful study of the absciss\ae\ of convergence of the
$\zeta_{\mathbf G(\mathcal O_v)}$ as a function of $v$ allowed Avni to
prove, in~\cite{avni:rational}, that $\zeta_\Gamma$ has a rational
abscissa of convergence; though its precise value is still mysterious.

The local factors $G=\mathbf G(\mathcal O_v)$ are compact $p$-adic
analytic groups, and Jaikin-Zapirain shows
in~\cite{jaikin-zapirain:zetapadic} that the representation zeta
function of such a group may be written as
\[\zeta_G(s)=\sum_{i=1}^k n_i^{-s} f_i(p^{-s})\]
for natural numbers $n_1,\dots,n_k$ and rational functions
$f_1,\dots,f_k\in\Q(p^{-s})$.


\subsection{``Quoi de neuf, docteur?''}
Here is a quick summary of the main differences between this article
and~\cite{bartholdi-h:zetawreath}.

Firstly, Isaacs' notion of ``character triples'' is fundamental to the
calculations done here. I found it necessary to express character
triples slightly differently, by making explicit a marking with a
given finite group. This makes also more transparent the extent to
which character triples are convenient computational tools to study
and manipulate cohomological information. Thus while character triples
are triples $(\chi,N,G)$ with $\chi\in\widehat N$ and $N\triangleleft
G$, I prefer to fix a group $B$, and call \emph{$B$-character triple}
a pair $(\chi,f)$ with $\chi\in\widehat{\ker f}$ and $f$ a
homomorphism to $B$. One recovers the classical notion by taking for
$f$ the natural map $G\mapsto G/N$.

Secondly, I associate a \emph{branch structure} to a branched group
$G$. This is a data structure made of a finite group $B$, a subgroup
$B_+$ of $B\wr_X Q$, and a surjective map $B_+\twoheadrightarrow B$.
It seems to capture in an efficient manner the important properties of
a branched group. The group $G$ itself is not determined by the branch
structure, but one may construct out of the branch structure a
profinite group $G(B)$ with a canonical map $G\to G(B)$.

\subsection{Acknowledgments}
I am grateful to Marty Isaacs for an enlightening comment on the
isotropy of induced representations, to Pierre de la Harpe for helpful
comments on earlier installments of the text, to Patrick Neumann for
help with Lemma~\ref{lem:extend}, to Joerg Br\"udern for references on
Tauberian theorems, and to the referee for his/her thoughtful remarks.

\section{Illustrations}\label{ss:illustrations}
I describe here some examples of self-similar branched groups, and
some information on their representation zeta functions. Let us start
by the precise definition of self-similar groups that we will use.
The definition of branched groups will appear in~\S\ref{ss:branch}.

\begin{defn}\label{def:ssg}
  A \emph{self-similar group} is a group $G$ endowed with an injective
  homomorphism $\psi\colon G\to G\wr_X Q$, for a permutation group $Q$
  acting on a finite set $X$. The map $\psi$ is called a
  \emph{self-similarity structure}, and the integer $\#X$ is called
  its \emph{degree}. Usually, the self-similarity structure is
  implicit, and one simply denotes by $G$ the self-similar group.
\end{defn}
The notation $\pair<g_1,\dots,g_d>q$ refers to the element of
$G\wr_X Q$ with $(g_1,\dots,g_d)\in G^X$ and $q\in Q$.

So as to avoid degenerate cases, we make the following restriction:
\begin{defn}\label{def:essg}
  An \emph{effective self-similar group} is a self-similar group whose
  branch structure satisfies the following conditions:
  \begin{enumerate}
  \item the degree $\#X$ is at least $2$;
  \item the action of $Q$ on $X$ is transitive;
  \item the projection $\psi(G)\to Q$ is surjective, and for each
    $x\in X$, the projection $\psi(G)\cap G^X\to G$ on co\"ordinate
    $x$ is surjective.\qedhere
  \end{enumerate}
\end{defn}
The second condition could, in fact, be relaxed to the requirement
that $Q$ act without fixed points on $X$. The third condition may be
ensured by replacing $Q$ by the image of $\psi$ and/or replacing $G$
by the projection of $\psi(G)\cap G^X$ to a co\"ordinate (possibly
after post-composing the self-similarity structure by an automorphism
of $G\wr_X Q$). All self-similar groups in this text are assumed to
be effective.

The map $\psi$ can be applied diagonally to all entries in $G^X$,
yielding a map $G^X\to(G\wr_X Q)^X$, and therefore a map $G\wr_X
Q\to(G\wr_X Q)\wr_X Q= G\wr_{X\times X}(Q\wr_X Q)$; more generally, we
write $\wr_X^nQ$ for the iterate $Q\wr_X\cdots\wr_X Q$, and get maps
$G\wr_{X^n}(Q\wr_X^n Q)\to G\wr_{X^{n+1}}(Q\wr_X^{n+1}Q)$ which we all
denote by $\psi$. We may compose these maps, and write $\psi^n$ for
the iterate $\psi^n\colon G\to G\wr_{X^n}(Q\wr_X^n Q)$.

By projecting to the permutation part, we then have homomorphisms
$G\to\sym{X^n}$ for all $n\in\N$ and, assembling these homomorphisms
together, we get a permutational action of $G$ on
$X^*=\bigsqcup_{n\ge0}X^n$; one may identify $X^*$ with the vertex set
of a rooted $\#X$-regular tree, by connecting $v_1\dots v_n$ to
$v_1\dots v_nv_{n+1}$ for all $v_i\in X$. In this manner, $G$ acts by
graph isometries. This action need not be faithful; if it is, then $G$
is called a \emph{faithful self-similar group}. In first three
examples below, this action is faithful; while in the fourth it is
not.

\begin{lem}
  Let $G$ be an effective self-similar group. Then its action on $X^n$
  is transitive for all $n\in\N$. In particular, $G$ is infinite.
\end{lem}
\begin{proof}
  We proceed by induction, the case $n=1$ being given by the second
  condition. Then, assuming that the action of $G$ is transitive on $X^n$,
  it follows from the third condition that the action of $\psi(G)\cap
  G^X$ on $x X^n$ is transitive for all $x\in X$, so the orbits of
  $\psi(G)\cap G^X$ are precisely $\{x X^n\}_{x\in X}$. Now since
  $\psi(G)$ maps onto $Q$ which is transitive, these orbits form a
  single $G$-orbit on $X^{n+1}$. Infiniteness of $G$ follows from the
  first assertion.
\end{proof}

The examples of groups that we consider below will be described by the
following data: a finite group $Q$, a finite $Q$-set $X$, a finitely
presented group $F$, and a homomorphism
$\tilde\psi\colon F\to F\wr_X Q$. Define normal subgroups of $F$ by
$R_0=1$ and $R_{n+1}=\tilde\psi^{-1}(R_n^X)$ for all $n\ge0$. The
\emph{injective quotient} of $F$ is by definition the self-similar
group $G:=F/\bigcup_{n\ge0}R_n$. The homomorphism $\tilde\psi$
descends to an injective map $\psi\colon G\hookrightarrow G\wr_X Q$.

\subsection{The Al\"eshin and Grigorchuk groups}\label{ss:gg}
The Grigorchuk group is obtained as follows. The cyclic group of order
$2$ is written $C_2$. Set
\[F=\langle a,b,c,d\mid a^2,b^2,c^2,d^2,bcd\rangle=C_2*(C_2\times C_2),\]
and define $\tilde\psi\colon F\to F\wr C_2$ by
\[\tilde\psi(a)=\pair<1,1>(1,2),\quad\tilde\psi(b)=\pair<a,c>,\quad\tilde\psi(c)=\pair<a,d>,\quad\tilde\psi(d)=\pair<1,b>.\]

Let $G$ be the injective quotient of $F$. It acts faithfully on
$\{1,2\}^*$.  A related group (see below) was first considered by
Al\"eshin in~\cite{aleshin:burnside}, providing a ``tangible'' example
of an infinite, finitely generated, residually finite, torsion group
(the first examples of groups with these properties are due to
Golod~\cite{MR28:5082}). Grigorchuk proved in~\cite{grigorchuk:growth}
that $G$'s word growth is strictly between polynomial and
exponential. See~\cite{harpe:ggt}*{Chapter~VIII} for an elementary
introduction to $G$. For its structure as a branched group,
see~\S\ref{ss:branch:gg}.

Since $G$ is a $2$-group, all its irreducible representations are
$2^n$-dimensional for some $n$; therefore $\zeta_G(s)=f(2^{-s})$ for a
power series $f\in\N[[2^{-s}]]$. Let us write $q=2^{-s}$ for brevity;
then the first values are
\begin{multline*}
  f(q)=8 + 10q + 29q^2 + 100q^3 + 413q^4 + 1990q^5 + 9787q^6 +
  50810q^7 + 278797q^8\\
  {} + 1593796q^9 + 9572828q^{10} + 60125360q^{11} + 396548538q^{12}\\
  {} + 2732836832q^{13} + 19674348692q^{14} + 147148989714q^{15} +\dots
\end{multline*}
and, for illustration, there are
$5554240222\cdots8648974784\approx5.5\cdot10^{93}$ irreducible
representations of degree $2^{100}$. This calculation took 4 minutes
on a 2010 laptop using~\textsc{Gap} and the author's
package~\textsc{Fr}. The functional equation involves $62$ variables
$\zeta_1,\dots,\zeta_{62}$.

The abscissa of convergence of $\zeta_G$ is computed as described
in~\S\ref{ss:using}, and is $\sigma_0(G)\approx3.293330470$.

Here is a brief description of Al\"eshin's group $\tilde G$ and its
relation to $G$. The Al\"eshin group can be viewed as a group acting
on $\{1,2\}^*$, generated by two elements $A,B$. The
recursions defining the generator's actions are
\[\psi^2(A)=\pair<\pair<a,c>,\pair<1,d>>,\qquad\psi^2(B)=\pair<\pair<1,1>,\pair<1,1>(1,2)>(1,2).\]

\begin{lem}\label{lem:GAcommensurable}
  The groups $G$ and $\tilde G$ have a common finite-index subgroup.
\end{lem}
\begin{proof}
  Consider the normal closure $\tilde G_0$ of $A$ in $\tilde
  G$. Clearly $\tilde G_0$ has index $4$ in $\tilde G$, and the
  generators of $\tilde G_0$ are involutions. The derived subgroup
  $\tilde G_0'$ therefore has finite index in $\tilde G$. Now
  $\psi^2(\tilde G_0')$ contains
  \[\psi^2([A,A^B])=[\pair<\pair<a,c>,\pair<1,d>>,\pair<\pair<d,1>,\pair<a,c>>]=\pair<\pair<[a,d],1>,\pair<1,1>>,
  \]
  so it contains $\pair<\pair<L,L>,\pair<L,L>>$ for the subgroup
  $L=\langle[a,d]\rangle^G$ of $G$. A direct computation shows that
  $L$ has index $32$ in $G$. Therefore, $L$ and $K$ have a common
  finite-index subgroup, so all of $\tilde G$,
  $\psi^{-2}\pair<\pair<L,L>,\pair<L,L>>$,
  $\psi^{-2}\pair<\pair<K,K>,\pair<K,K>>$, $K$ and $G$ have a common
  finite-index subgroup.
\end{proof}

It was already shown in~\cite{kargapolov-merzlyakov:otg}*{page 229}
that $G$ is a section of $\tilde G$; they poetically describe the
extraction of $G$ from $\tilde G$ as ``tearing off Adam's rib''.

\begin{cor}
  The representation zeta functions of $G$ and $\tilde G$ have the
  same absciss\ae\ of convergence.
\end{cor}
\begin{proof}
  By Lemma~\ref{lem:GAcommensurable}, the groups $G$ and $\tilde G$
  are commensurable. For two Dirichlet series $\eta(s)=\sum a_n
  n^{-s}$ and $\theta(s)=\sum b_n n^{-s}$, let us write
  $\eta\le\theta$ to mean $\sum_{j\le n}a_j\le\sum_{j\le n}b_j$ for
  all $n\in\N$. It follows
  from~\cite{lubotzky-martin:polynomialrepgrowth}*{Lemma~2.2} that if
  $G,H$ are groups and $H$ is a finite-index subgroup of $G$, then
  \begin{equation}\label{eq:index}
    \zeta_H(s)\le [G:H]^{1+s}\zeta_G(s)\text{ and }\zeta_G(s)\le[G:H]\zeta_H(s),
  \end{equation}
  so $\zeta_H$ and $\zeta_G$ have the same domain of convergence.
\end{proof}


\subsection{The Gupta-Sidki group}\label{ss:gs}
The Gupta-Sidki groups are obtained as follows. The cyclic group of
order $p$ is written $C_p$. For each prime $p\ge3$, set
\[F_p=\langle a,t\mid a^p,t^p\rangle=C_p*C_p,\]
and define $\tilde\psi\colon F_p\to F_p\wr C_p$ by
\[\tilde\psi(a)=\pair<1,\dots,1>(1,\dots,p),\quad\tilde\psi(t)=\pair<a,a^{-1},1,\dots,1,t>.\]
Let $G_p$ be the injective quotient of $F_p$. It acts faithfully on
$\{1,\dots,p\}^*$.

These groups $G_p$ are shown in~\cite{gupta-s:burnside} to be
infinite, finitely-generated torsion $p$-groups.  For their structure
as branched groups, see~\S\ref{ss:branch:gs}. The study of their
representations was initiated by Passman and
Temple~\cite{passman-t:reps}; their main result, in the present
paper's language, is $\sigma_0(G_p)\ge p-2$.

We restrict our consideration to the case $p=3$.  Since $G_3$ is a
$3$-group, all its irreducible representations are $3^n$-dimensional
for some $n$; therefore $\zeta_{G_3}(s)=f(3^{-s})$ for a power series
$f\in\Z[[3^{-s}]]$. Writing $q=3^{-s}$, the first values are
\begin{multline*}
  f(q)=9 + 26q + 402q^2 + 6876q^3 + 178160q^4 + 7527942q^5 + 461931336q^6\\
  {} + 31704156696q^7 + 2421457788330q^8\\
  {} + 197775615899520q^9 + 16915932297409064q^{10} + \dots
\end{multline*}
and there are $1386068855\dots8306590020\approx1.3\cdot10^{96}$
representations of degree $3^{50}$. This calculation took 6 seconds on
a 2010 laptop using~\textsc{Gap} and the author's
package~\textsc{Fr}. The functional equation involves $8$
variables. It may be written in the slightly simplified form as
{\small\begin{align*}
  \zeta_{G_3}(s) = {} & \tfrac19q^2\zeta_1(s)+q\zeta_2(s) + q\zeta_3(s)
                   + 2q\zeta_4(s) + (9+2q)\zeta_6(s),\\
\zeta_1(s) = {} & \tfrac19q^2\zeta_1(s)^3
 + \tfrac13q^2\zeta_1(s)^2\zeta_2(s)
 + \tfrac13q^2\zeta_1(s)^2\zeta_3(s)
 + \tfrac23q^2\zeta_1(s)^2\zeta_4(s)
 + q^2\zeta_1(s)^2\zeta_6(s)\\
&+ \tfrac13q^2\zeta_1(s)\zeta_2(s)^2
 + \tfrac23q^2\zeta_1(s)\zeta_2(s)\zeta_3(s)
 + \tfrac43q^2\zeta_1(s)\zeta_2(s)\zeta_4(s)
 + 2q^2\zeta_1(s)\zeta_2(s)\zeta_6(s)\\
&+ \tfrac13q^2\zeta_1(s)\zeta_3(s)^2
 + \tfrac43q^2\zeta_1(s)\zeta_3(s)\zeta_4(s)
 + 2q^2\zeta_1(s)\zeta_3(s)\zeta_6(s)
 + q^2\zeta_1(s)\zeta_4(s)^2\\
&+ 2q^2\zeta_1(s)\zeta_4(s)\zeta_6(s)
 + q\zeta_2(s)^3
 + \tfrac13q^2\zeta_2(s)^2\zeta_3(s)
 + \tfrac23q^2\zeta_2(s)^2\zeta_4(s)\\
&+ 9q\zeta_2(s)^2\zeta_6(s)
 + \tfrac43q^2\zeta_2(s)\zeta_3(s)\zeta_4(s)
 + \tfrac23q^2\zeta_2(s)\zeta_4(s)^2
 + 18q\zeta_2(s)\zeta_6(s)^2\\
&+ \tfrac19q^2\zeta_3(s)^3
 + \tfrac23q^2\zeta_3(s)^2\zeta_4(s)
 + (3q+\tfrac23q^2)\zeta_3(s)\zeta_4(s)^2
 + 18q\zeta_3(s)\zeta_4(s)\zeta_6(s)\\
&+ 18q\zeta_3(s)\zeta_6(s)^2
 + \tfrac29q^2\zeta_4(s)^3
 + 6q\zeta_4(s)^2\zeta_6(s)
 + 36q\zeta_4(s)\zeta_6(s)^2\\
&+ 72\zeta_6(s)^3
 - q^2\zeta_1(3s)
 - 9q\zeta_2(3s)
 - q^2\zeta_3(3s)
 - 2q^2\zeta_4(3s)
 - 18\zeta_6(3s)\\
&= 54+\mathcal O(q),\\
\zeta_2(s) = {} & \tfrac13q^2\zeta_1(s)\zeta_4(s)^2
 + 2q^2\zeta_1(s)\zeta_4(s)\zeta_6(s)
 + 3q^2\zeta_1(s)\zeta_6(s)^2
 + \tfrac13q^2\zeta_2(s)\zeta_3(s)^2\\
&+ 2q^2\zeta_2(s)\zeta_3(s)\zeta_6(s)
 + \tfrac23q^2\zeta_2(s)\zeta_4(s)^2
 + 4q^2\zeta_2(s)\zeta_4(s)\zeta_6(s)
 + 9q\zeta_2(s)\zeta_6(s)^2\\
&+ q^2\zeta_3(s)^2\zeta_6(s)
 + \tfrac13q^2\zeta_3(s)\zeta_4(s)^2
 + 2q^2\zeta_3(s)\zeta_4(s)\zeta_6(s)
 + 9q\zeta_3(s)\zeta_6(s)^2\\
&+ \tfrac23q^2\zeta_4(s)^3
 + (3q+3q^2)\zeta_4(s)^2\zeta_6(s)
 + 18q\zeta_4(s)\zeta_6(s)^2
 + (9+18q)\zeta_6(s)^3
 - 3\zeta_6(3s)\\
&= 6+\mathcal O(q),\\
\zeta_3(s) = {} & q^2\zeta_1(3s)
 + 3q\zeta_2(3s)
 + q^2\zeta_3(3s)
 + 2q^2\zeta_4(3s)
 + 6\zeta_6(3s) = 6+\mathcal O(q),\\
\zeta_4(s) = {} & 3q\zeta_2(3s) + 6\zeta_6(3s) = 6+\mathcal O(q),\\
\zeta_6(s) = {} & \zeta_6(3s) = 1.
\end{align*}}

The abscissa of convergence of $\zeta_{G_3}$ is computed as described
in~\S\ref{ss:using}, and is $\sigma_0(G_3)\approx4.250099133$. In view
of the Passman-Temple result mentioned above, it would be interesting
to examine the dependency of $\sigma_0(G_p)$ on $p$.

\subsection{Wreath products}
There exist sundry residually-finite, finitely generated groups that
are isomorphic to their wreath product with a non-trivial finite
group; here is such an example. Set
\[F=A_5 * A_5,\] with $A_5$ the alternating group on five letters, and
distinguish both copies of $A_5$ by writing `$\overline a$' for
permutations in the second copy. Set $X=\{1,\dots,5\}$, and define
$\tilde\psi\colon F\to F\wr_X A_5$ by
\[\tilde\psi(a)=\pair<1,\dots,1>a,\quad\tilde\psi(\overline a)=\pair<\overline a,a,1,1,1>.\]
Let $W$ be the injective quotient of $F$; it acts faithfully on $X^*$.

This example was considered, among others,
in~\cite{bartholdi-h:zetawreath}*{Example 4}; it is a branched group,
and more precisely $\psi$ is an isomorphism. The representation zeta
function of $W$ starts as
\begin{multline*}
  \zeta_W(s)=1+2\cdot3^{-s}+4^{-s}+5^{-s}+6\cdot15^{-s}+3\cdot20^{-s}+3\cdot25^{-s}+2\cdot45^{-s}\\
  {}+60^{-s}+19\cdot75^{-s}+4\cdot90^{-s}+9\cdot100^{-s}+\cdots,
\end{multline*}
and has abscissa of convergence $\sigma_0(W)\approx1.17834859575464$,
computed as described in~\S\ref{ss:using}.

\subsection{Non-faithful self-similar groups}\label{ss:nonfaithful}
The group $W$ acts on the tree $X^*$, and therefore on its boundary
$X^\infty$. Consider the ray $\rho=1^\infty$ in it, and its orbit
$\mathcal O$ in $X^\infty$. Consider then the permutational wreath
product $G := C_2\wr_{\mathcal O}W$. This group is also self-similar;
to see that, consider now
\[F=\langle A_5, \overline{A_5}, s\mid s^2,[s,\overline a]\text{ for
  all }\overline a\in\overline{A_5}\rangle,\]
extend $\tilde\psi$ by
\[\tilde\psi(s)=\pair<s,1,1,1,1>,\]
and let $G$ be the injective quotient of $F$. Remark that $s$ acts
trivially on $X^*$, so that $G$ does not act faithfully on $X^*$. The
group $G$ is also branched, see~\S\ref{ss:branch:nonfaithful}. The
zeta function of $G$ starts as
\begin{multline*}
  \zeta_G(s)= 2+4\cdot3^{-s}+2\cdot4^{-s}+8\cdot5^{-s}+4\cdot10^{-s}+26\cdot15^{-s}+14\cdot20^{-s}+48\cdot25^{-s}+8\cdot45^{-s}\\
  {}+24\cdot50^{-s}+28\cdot60^{-s}+172\cdot75^{-s}+12\cdot80^{-s}+24\cdot90^{-s}+132\cdot100^{-s}+\cdots,
\end{multline*}
and has abscissa of convergence $\sigma_0(G)\approx1.64046292658488$,
as follows from~\S\ref{ss:using}.

\section{Representations of extensions}\label{ss:clifford}
I recall Clifford's construction of representations of an
extension. First, a \emph{linear representation} of a group $G$ is a
homomorphism $\rho\colon G\to \GL_n(\C)$. Two linear representations
$\rho,\rho'\colon G\to \GL_n(\C)$ are \emph{equivalent}, written
$\sim$, if there exists $T\in \GL_n(\C)$ such that
$\rho(g)T=T\rho'(g)$ for all $g\in G$.

A \emph{projective representation} of a group $G$ is a homomorphism
$\rho\colon G\to \PGL_n(\C):=\GL_n(\C)/\C^\times$. Two projective
representations $\rho,\rho'$ are \emph{equivalent} if there exists
$T\in \PGL_n(\C)$ such that $\rho(g)T=T\rho'(g)$ for all $g\in G$.

Let $\rho$ be a linear or projective representation, to $\GL_n(\C)$ or
$\PGL_n(\C)$. Its \emph{degree} $\deg(\rho)$ is $n$. The
\emph{contragredient} representation $\rho^\vee$ is defined by
$\rho^\vee(g)=\rho(g^{-1})^*$, the matrix adjoint. For linear
representations $\rho,\sigma$ of degree $m,n$ respectively, the
\emph{tensor product} $\rho\otimes\sigma$ is the linear representation
$g\mapsto\rho(g)\otimes\sigma(g)$ into $\GL_{mn}(\C)$; and if
$\rho,\sigma$ are both projective representations, their tensor
product is a projective representation into $\PGL_{mn}(\C)$.

Let $\rho\colon G\to \PGL_n(\C)$ be a projective representation. Choose a
lift $\tilde\rho\colon G\to \GL_n(\C)$. Define then $\tilde
c_{\tilde\rho}\colon G\times G\to\C^\times$ by $\tilde
c_{\tilde\rho}(g,h)=\tilde\rho(g)\tilde\rho(h)/\tilde\rho(gh)$. A
quick calculation shows that $\tilde c_{\tilde\rho}$ satisfies the
$2$-cocycle identity
\[\tilde c_{\tilde\rho}(g,h) \,/\, \tilde c_{\tilde\rho}(g,hk)\times \tilde c_{\tilde\rho}(gh,k) \,/\, \tilde c_{\tilde\rho}(h,k)=1,\]
and therefore defines a cohomology class $c_\rho$ in
$H^2(G,\C^\times)$, which depends on $\rho$ only, and not on the
choice of lift $\tilde\rho$.

\subsection{Exact sequences} Let now
\[1\longrightarrow N\longrightarrow G\overset f\longrightarrow Q\longrightarrow 1
\]
be an exact sequence. If $\rho$ be a representation (linear or
projective) of $N$, its \emph{inertia} is the group $G_\rho=\{g\in
G\mid {}^g\!\rho\sim\rho\}$ consisting of those $g\in G$ such that the
conjugate representation ${}^g\!\rho\colon n\mapsto\rho(n^g)$ is equivalent to
$\rho$. The representation $\rho$ is said to be \emph{inert in $H$}
whenever $H\le G_\rho$.

Assume now that $\rho$ is an irreducible, degree-$n$ linear
representation of $N$. Then $\rho$ extends to a unique projective
representation $\overline\rho$ of $G_\rho$, as follows. Fix a right
transversal $X$ of $N$ in $G_\rho$. For each $x\in X$, choose $T_x\in
\GL_n(\C)$ such that $T_x\rho(h^x)=\rho(h)T_x$ for all $h\in N$; this
$T_x$ is unique up to scalars, by Schur's Lemma.  For $g=hx\in
G_\rho$, set $\tilde\rho(g)=\rho(h)T_x$, and let $\overline\rho(g)$ be
$\tilde\rho(g)$'s image in $\PGL_n(\C)$. Then, since the $T_x$ are
uniquely determined, $\overline\rho$ is a projective
representation. Furthermore, the $2$-cocycle $\tilde c_{\tilde\rho}$
vanishes on $N\times N$, so defines a cohomology class $c_\rho\in
H^2(G_\rho/N,\C^\times)$.

Let $\chi$ be an irreducible projective representation of $G_\rho/N$
with cohomology class $c_\rho^{-1}$; then
$\overline\rho\otimes(\chi\circ f)$ is a projective representation of
$G_\rho$ with trivial cohomology class. Say $\chi$ is of degree $m$,
and let $\tilde\chi$ be a lift $G_\rho/N\to\GL_m(\C)$ of $\chi$; then
$\tilde\rho\otimes(\tilde\chi\circ f)$ is a lift of
$\overline\rho\otimes(\chi\circ f)$, so its $2$-cocycle is a
coboundary, namely the $2$-cocycle $(\delta b)(g,h)=b(g)b(h)/b(gh)$
associated with a function $b\colon G_\rho/N\to\C^\times$. Furthermore, $b$
is unique up to multiplication by a homomorphism $\mu\in
H^1(G_\rho/N,\C^\times)$. Then
$g\mapsto\tilde\rho(g)\otimes\tilde\chi(f(g))/b(g)$ is a linear
representation of $G_\rho$, which we denote by
$\sigma'_{\rho,\chi}$.

We call such $\sigma'_{\rho,\chi}$ \emph{extensions} of $\rho$; they
are irreducible representations whose restriction to $N$ is a direct
sum of copies of $\rho$.  Finally, let $\sigma_{\rho,\chi,\mu}$ be the
induced representation of $\sigma'_{\rho,\chi}\otimes\mu$ up to $G$.

\begin{thm}[Clifford~\cite{clifford:induced}]\label{thm:clifford}
  With the notation above, $\sigma_{\rho,\chi,\mu}$ is an irreducible
  representation of $G$, and every irreducible representation of $G$
  is equivalent to some $\sigma_{\rho,\chi,\mu}$.

  The multiplicity of $\sigma_{\rho,\chi,\mu}$ in that list behaves as
  follows: for a group $Q$ and a class $c\in H^2(Q,\C^\times)$, denote
  by $\widehat Q^c$ the set of equivalence classes of projective
  representations of $Q$ with cocycle $c$; then the correspondence
  $(\rho,\chi,\mu)\mapsto\sigma_{\rho,\chi,\mu}$ is a map
  \[\sigma\colon\bigsqcup_{\rho\in\widehat N}\left(\widehat{G_\rho/N}^{c_\rho^{-1}}\times H^1(G_\rho/N,\C^\times)\right)\to\widehat G\]
  which is surjective, and such that every $\sigma_{\rho,\chi,\mu}$
  has $\#H^1(G_\rho/N,\C^\times)\cdot[G:G_\rho]$ preimages.\qed
\end{thm}

We will need to understand how the inertia
subgroup changes under extension. I state the following property as a
general lemma:
\begin{lem}\label{lem:extend}
  Let $G$ be a group with normal subgroup $N$; let $\rho$ be a
  representation of $N$. Consider a subgroup $H$ with $N\le H\le
  G_\rho$. Let $\sigma$ be an extension of $\rho$ to $H$. Then
  $G_\sigma\le G_\rho$.
\end{lem}
\begin{proof}
  Since $\sigma$ is an extension of $\rho$ and $\rho$ is inert in $H$,
  the restriction of $\sigma$ to $N$ is a direct sum of $[H:N]$ copies
  of $\rho$. Consider $g\in G_\sigma$, and write $T_g$ as a
  $[H:N]\times[H:N]$ block matrix. Then
  $(T_g)_{ij}\rho(n^g)=\rho(n)(T_g)_{ij}$ for all
  $i,j\in\{1,\dots,[H:N]\}$; and since $T_g$ is invertible, the
  $(T_g)_{ij}$ span $M_n(\C)$ so some linear combination $U_g$ of them
  is invertible; then $U_g\rho(n^g)=\rho(n)U_g$ so $g\in G_\rho$.
\end{proof}

%
%

\section{Representation triples}
I recall Isaacs' notion of \emph{character triple}, with a slightly
different notation. See also~\cite{jaikin-zapirain:zetapadic}*{\S5}
for a more modern formulation.

\begin{defn}
  Let $B$ be a finite group. A \emph{$B$-representation triple} is a
  pair $\Theta=(\rho,f)$, with $f\colon G\to B$ a homomorphism with
  kernel $N$ and $\rho\in\widehat N$ a representation that is inert in
  $G$.
\end{defn}
(The reader may wonder why they are called triples and not
pairs. Isaacs' original definition involves triples $(\chi,N,G)$ with
$\chi$ an $N$-character that is inert in $G$. We explicitly add a
marking by a group $B$ to the data, and remove $B$, $G$ and $N$ from
the notation.)

We introduce the following terminology: for a $B$-representation
triple $\Theta=(\rho,f)$, its \emph{source} is
$\src(\Theta):=\src(f):=G$; its \emph{image} is
$\img(\Theta):=\img(f):=f(G)\le B$; its \emph{representation} is
$\rho(\Theta):=\rho$; its \emph{marking} is $f(\Theta):=f$.  If
$\Theta=(\rho,f)$, we also define $\Theta^\vee=(\rho^\vee,f)$ the
triple with same marking but contragredient representation.

A morphism between two $B$-representation triples $(\rho,f)$ and
$(\rho',f')$ is a map $\sigma\colon\src(f)\to \src(f')$ such that
$f'=f\circ\sigma$ and $\rho\sim\rho'\circ\sigma$. There is also a
weaker notion than isomorphism of $B$-representation triples, that of
\emph{equivalence}, which we describe now.

For $G$ a group with normal subgroup $N$ and $\rho\in\widehat N$, let
$\mathcal R(G|\rho)$ denote the monoid of representations of $G$ whose
restriction to $N$ is a multiple of $\rho$. It is an abelian monoid,
freely generated by the irreducible representations of $G$ that
restrict to a multiple of $\rho$, and admits a scalar product
$\langle\mid\rangle$ making the irreducible representations an
orthonormal basis.

\begin{defn}[Essentially~\cite{isaacs:ctfg}*{Definition~11.23}]\label{def:equiv}
  Two $B$-representation triples $(\rho,f)$ and $(\rho',f')$ are
  \emph{equivalent} if $\img(f)=\img(f')$ and for every $H\le \img(f)$
  there exists an isometry
  \[\sigma_H\colon\mathcal R(f^{-1}(H)|\rho)\to\mathcal R((f')^{-1}(H)|\rho')\]
  such that, for every $N\le H\le\img(f)$ and every $\chi\in\mathcal
  R(f^{-1}(H)|\rho)$, we have
\[\sigma_{f^{-1}(N)}(\chi_{f^{-1}(N)})=(\sigma_{f^{-1}(H)}(\chi))_{(f')^{-1}(N)},\]
\[\sigma_{f^{-1}(H)}(\chi\otimes(\beta\circ f))=\sigma_{f^{-1}(H)}(\chi)\otimes(\beta\circ f')\text{ for all }\beta\in\widehat{\img(f)}.\qedhere\]
\end{defn}

Schur considered projective representations
in~\cites{schur:darstellung,schur:untersuchungen}. In modern language,
he showed that $H_2(G,\Z)$ is finite for every finite group $G$, and
that there exists at least one extension
\[1\longrightarrow H_2(G,\Z)\longrightarrow \tilde G\overset f\longrightarrow G\longrightarrow 1
\]
such that $H_2(G,\Z)$ is contained in $[\tilde G,\tilde G]$; this
implies in particular that the lift of any generating set of $G$ is a
generating set of $\tilde G$. One calls $\tilde G$ a \emph{Schur
  cover} of $G$, and the epimorphism $f$ a \emph{Schur covering
  map}\footnote{$\tilde G$ is sometimes called a \emph{stem cover}.}.

\begin{thm}[Isaacs,~\cite{isaacs:ctfg}*{Theorem~11.28}]\label{thm:isaacs}
  Every $B$-representation triple is equivalent to a
  $B$-representation triple $(\chi,f)$ with $f\colon\tilde H\to H\le B$ a
  Schur covering map, and $\chi\in\widehat{H_2(H,\Z)}=H^2(H,\C^\times)$.
\end{thm}

In particular, there are finitely many equivalence classes of
$B$-representation triples.  A $B$-representation triple
$\Theta=(\rho,f)$ is a convenient way of keeping track of a group
$\img(f)$ and a cohomology class in $H^2(\img(f),\C^\times)$.

The two procedures at the heart of Clifford's description
from~\S\ref{ss:clifford} --- extension and induction --- can be
rephrased in terms of representation triples.

Consider a $B$-representation triple $\Theta=(\rho,f)$, and a
homomorphism $g\colon B\to C$. Let $L$ denote the kernel of $g\circ f$; we have
$\ker(f)=N\le L\le G=\src(f)$. Let $\{\rho_1,\dots,\rho_n\}$ denote
those irreducible representations of $L$ that restrict on $N$ to a
multiple of $\rho$. For $i=1,\dots,n$, let $G_i$ denote the inertia of
$\rho_i$ in $G$. The $g$-\emph{extensions} of $\Theta$ are the
$C$-representation triples
$\Theta_1=(\rho_1,(g\circ f)|_{G_1}),\dots,\Theta_n=(\rho_n,(g\circ f)|_{G_n})$.

\begin{lem}
  The equivalence classes of the $C$-triples $(\Theta_i)_{1\le i\le
    n}$ depend only on the equivalence class of $\Theta$.
\end{lem}
\begin{proof}
  Follows immediately from Definition~\ref{def:equiv} and
  Lemma~\ref{lem:extend}.
\end{proof}

Note that extension of triples covers both extension and induction;
the induction is performed from $\ker(g)\cap\img(f)$ to $\ker(g)$, or,
equivalently, from $\img(f)$ to $\img(f)\ker(g)$, and in fact does not
modify the triple at all. This is seen as follows. Consider a
$B$-representation triple $\Theta=(\rho,f)$ with $\rho\in\widehat N$
and $f\colon G\to B$. Let $H,M$ be groups with $N\triangleleft G\le H,N\le
M\triangleleft H$ and $M\cap G=N$ and $MG=H$ and $H_\rho=G$. Then
$G/N\cong H/M$; define $h\colon H\to B$ by $h(xy)=f(y)$ for $x\in M,y\in G$;
this is well-defined because $M\cap G=N=\ker(f)$. Note
$\ker(h)=M$. Induce $\rho$ to $M$, and let $\Theta'$ be the
$B$-representation $(\rho^M,h)$.

\begin{lem}[see~\cite{jaikin-zapirain:zetapadic}*{Corollary~5.3}]
  The triples $\Theta$ and $\Theta'$ are equivalent.
\end{lem}
\begin{proof}
  Follows immediately from Definition~\ref{def:equiv}. The map
  $\sigma\colon\mathcal R(G|\rho)\to\mathcal R(H|\rho^M)$ is simply given
  by induction to $H$, namely $\chi\mapsto\chi^H$.
\end{proof}


We may deduce from Theorem~\ref{thm:clifford} a formula expressing the
representation zeta function of a group in terms of representations of
a normal subgroup.  Consider an exact sequence
  \[1\longrightarrow N\longrightarrow G\overset f\longrightarrow B\longrightarrow 1
\]
For a $B$-representation triple $\Theta$, define the Dirichlet series
\[\zeta_{G,\Theta}(s)=\sum_{\substack{\rho\in\widehat N\\ (\rho,f)\sim\Theta}}(\deg\rho)^{-s}.\]
\begin{prop}\label{prop:cliffordtriple}
  With the notation above,
  \[\zeta_G(s)=\sum_{\Theta\in\{\text{$B$-representation triples}\}}\zeta_{G,\Theta}(s)\zeta_{\Theta^\vee}(s)[B:\img(\Theta)]^{-1-s}.\]
\end{prop}
\begin{proof}
  Consider an irreducible representation $\rho$ of $N$ with character
  triple $\Theta$; such representations are counted by
  $\zeta_{G,\Theta}(s)$. According to Theorem~\ref{thm:clifford}, a
  representation of $G$ is obtained by tensoring $\rho$ with a
  representation $\chi$ of opposite cocycle, so as to obtain a linear
  representation of $\rho$'s inertia subgroup; such $\chi$ are counted
  by $\zeta_{\Theta^\vee}(s)$. This representation is then induced to
  a representation of $G$; induction increases the degree by
  $[B:\img(\Theta)]$, and yields $[B:\img(\Theta)]$ copies of the same
  representation of $G$.
\end{proof}

\section{Branched groups}\label{ss:branch}
We turn now to the notion of \emph{self-similar branched group}, presenting it
in a slightly more general and algebraic manner than is usual;
see~\cite{nekrashevych:ssg} or~\cite{bartholdi-g-s:bg} for classical
references.

Let $G$ be a self-similar group with self-similarity structure
$\psi\colon G\to G\wr_X Q$.
\begin{defn}\label{def:branch}
  The self-similar group $G$ is \emph{branched} if there exists a
  finite-index subgroup $K\le G$ such that $\psi(K)\ge K^X$. One says
  then that $G$ is branched \emph{over} $K$.
\end{defn}
The subgroup $K$ may be assumed to be normal; and in fact there exists
a maximal such $K$, because if $K_0,K_1$ both satisfy $\psi(K_i)\ge
K_i^X$ then $\langle K_0,K_1\rangle$ also satisfies that property.

For purposes of computation, it is useful to introduce a finite
structure capturing important features of branched groups.
\begin{defn}\label{def:branchstruc}
  A \emph{branch structure} is a pair $(B,\phi)$ such that
  \begin{enumerate}
  \item $B$ is a finite group;
  \item $\phi$ is an epimorphism from a subgroup $B_+$ of $B\wr_X Q$
    onto $B$.
  \end{enumerate}

  Let $G$ be a self-similar group. A \emph{branch structure for $G$}
  is a branch structure $(B,\phi)$ such that
  \begin{enumerate}
  \item there exists an epimorphism $f\colon G\twoheadrightarrow B$;
  \item denoting $f_1$ the natural map $f\wr1\colon G\wr_X Q\to B\wr_X
    Q$, we have $B_+=f_1\psi(G)$ and $f=\phi f_1\psi$:
    \[\begin{tikzcd} G\arrow[hookrightarrow]{r}{\psi}\arrow[twoheadrightarrow,swap]{d}{f} & \psi(G)\subseteq G\wr_X Q\arrow[twoheadrightarrow]{d}{f_1}\\
      B & \arrow[twoheadrightarrow]{l}{\phi} B_+\subseteq B\wr_X Q.
    \end{tikzcd}\qedhere\]
  \end{enumerate}
\end{defn}

\begin{lem}
  A self-similar group is branched if and only if it has a branch structure.
\end{lem}
\begin{proof}
  Assume first that $G$ is branched over its normal subgroup
  $K$. Define $B=G/K$ with natural map $f\colon G\to B$. Define then
  $f_1$ as in Definition~\ref{def:branchstruc}, and set
  $B_+=f_1\psi(G)$. Define finally $\phi\colon B_+\twoheadrightarrow B$ by
  $\phi(f_1(\psi(g)))=f(g)$. This map is well-defined because
  $K^X\le \psi(K)$.

  Conversely, if $(B,\phi)$ is a branch structure for $G$ then let $K$
  denote the kernel of a map $f\colon G\to B$ as in
  Definition~\ref{def:branchstruc}, and note that $G$ is branched over
  $K$.
\end{proof}
Note that, just as there exists a maximal subgroup $K$ in
Definition~\ref{def:branch}, there exists a minimal branch structure
$(B,\phi)$.

The branch structure captures all the information we will need of $G$,
so that we may forget $G$ altogether when we have its branch
structure. In fact, let $(B,\phi)$ be a branch structure for
$G$. Define then a sequence of groups $G_n$, with maps $\phi_n\colon
G_n\to G_{n-1}$, as follows: $G_{-1}=B$, $G_0=B_+$, $\phi_0=\phi$, and
$G_{n+1}=\{\pair<g_x>q\in G_n\wr_X Q\mid \pair<\phi_n(g_x)>q\in
G_n\}$, with $\phi_{n+1}(\pair<g_x>q)=\pair<\phi_n(g_x)>q$. Finally
form the inverse limit
\[G(B)=\varprojlim(G_n,\phi_n).\]

\begin{lem}
  If $(B,\phi)$ be a branch structure, then the group $G(B)$ is a
  profinite self-similar branched group, and $B$ is a branch structure
  for $G(B)$.

  If furthermore $(B,\phi)$ be a branch structure for $G$, then there
  exists a canonical map $\iota\colon G\to G(B)$ interlacing the
  self-similarity structures of $G$ and $G(B)$, and $\iota$ is
  injective if $G$ is faithful.
\end{lem}
\begin{proof}
  It is clear that $G(B)$ is profinite, being defined as a limit of
  finite groups.

  An element of $G(B)$ is a sequence $h=(\cdots\twoheadrightarrow
  h_n\twoheadrightarrow h_{n-1}\cdots)$, with $h_n\in G_n$, namely
  $h_n=\pair<g_{n,x}>q_n$, with $q_n=q$ for all $n\ge0$. Define
  $\psi(h)=\pair<(\cdots\twoheadrightarrow g_{n,x}\twoheadrightarrow
  g_{n-1,x})_x>q$. This shows that $G(B)$ is self-similar.

  We next show that $B$ is a branch structure for $G(B)$. Projection
  on the last group $G_{-1}$ defines a homomorphism $G(B)\to B$, and
  $\psi(G(B))$ projects to $G_0\subseteq B\wr_X Q$.

  Suppose finally that $B$ is a branch structure for the self-similar
  branched group $G$ with self-similarity structure $\psi\colon G\to G\wr_X
  Q$. Define inductively maps $\iota_n\colon G\to G_n$ by
  $\iota_{-1}=f$ and $\iota_n(g)=\pair<\iota_{n-1}(g_x)>q$ if
  $\psi(g)=\pair<g_x>q$, for all $n\ge0$. Then
  $\phi_n\circ\iota_n=\iota_{n-1}$ for all $n\ge0$, so the maps
  $\iota_n$ assemble into a map $\iota\colon G\to\varprojlim G_n$.

  If $G$ is faithful, then $\bigcap_{n\ge0}\psi^{-n}(K^{X^n})=1$, so
  $\bigcap_{n\ge0}\ker(\iota_n)=1$ and $\iota$ is injective.
\end{proof}

Note then that $G(B)$ defines a topology on $G$, which is intermediate
between the congruence topology (in which neighbourhoods of the
identity are stabilizers of large subtrees) and the profinite
completion (in which every finite-index subgroup is a
neighbourhood). This topology is Hausdorff precisely when $G$ is
faithful. See~\cite{bartholdi-s-z:cspbg} for details on these
topologies.

\begin{prop}\label{prop:[K,K]<kernel}
  Let $G$ be a self-similar branched group over $K$, and let
  $\rho\colon G\to\GL_n(\C)$ be a linear representation of $G$. Then,
  for all $\ell\in\N$ large enough depending only on $n$, the kernel
  of $\rho$ contains $\psi^{-\ell}([K,K]^{X^\ell})$.
\end{prop}
\begin{proof}
  Assume $K\neq1$, otherwise there is nothing to show. The image of
  $\psi^{-\ell}(K^{X^\ell})$ in $\GL_n(\C)$ has bounded rank, so that
  there exists a constant $b$, depending only on $n$, with the
  following property: for all $\ell$ there exists a subset
  $\Omega\subseteq X^\ell$, with $\#(X^\ell\setminus\Omega)\le b$,
  such that $\ker\rho\cap\psi^{-\ell}(K^{X^\ell})$ maps onto
  $\psi^{-\ell}(K^\Omega)$. In particular, for $\ell\gg0$ one has
  $\Omega\neq\emptyset$, say $\omega\in\Omega$; then
  $[\ker\rho,\psi^{-\ell}(1\times\cdots\times
  K\times\cdots\times1)]=\psi^{-\ell}(1\times\cdots\times[K,K]\times\cdots\times1)$,
  with the non-trivial entry each times in position $\omega$. Since
  the action of $G$ on $X^n$ is transitive, we get
  $\psi^{-\ell}([K,K]^{X^\ell})\le\ker\rho$.
\end{proof}

\begin{cor}\label{cor:rigid}
  Let $G$ be a self-similar group, branched over $K$. Then $G$ is
  rigid if and only if $K/[K,K]$ is finite.
\end{cor}
\begin{proof}
  If $K/[K,K]$ is infinite, then it has infinitely many irreducible
  $1$-dimensional representations, so $G/[K,K]$ has infinitely many
  representations of degree at most $[G:K]$.

  Conversely, assume $K/[K,K]$ is finite, and consider $n\in\N$. By
  Proposition~\ref{prop:[K,K]<kernel}, there exists $\ell\in\N$ such
  that all $n$-dimensional representations of $G$ factor through
  $G/\psi^{-\ell}([K,K]^{X^\ell})$, which is finite; so there are
  finitely many $n$-dimensional representations.
\end{proof}

\begin{rem}\label{rem:[K,K]}
  In case $[K,K]$ contains $\psi^{-\ell}(K^{X^\ell})$ for some
  $\ell\in\N$, then the sharper statement holds that every linear
  representation has kernel containing $\psi^{-\ell}(K^{X^\ell})$ for
  some $\ell\in\N$.
\end{rem}

\begin{rem}\label{rem:zetaG(B)}
  If the self-similar group $G$ is branched over $K$, then it is also
  branched over $[K,K]$, so that there exists a branch structure with
  $B=G/[K,K]$ and with the additional property that every linear
  representation $\rho\colon G\to\GL_n(\C)$ factors through $G_\ell$
  for some $\ell$ large enough.

  Therefore, the representation zeta function of $G$ coincides with
  the zeta function counting all continuous representations of the
  profinite group $G(B)$.
\end{rem}

We now turn to the examples introduced in~\S\ref{ss:illustrations},
and describe their branch structures.

\subsection{The Grigorchuk group}\label{ss:branch:gg}
The maximal branching subgroup of the Grigorchuk group
(see~\S\ref{ss:gg}) is well-known; we recall it briefly.

In the Grigorchuk group $G$, consider the subgroup
$K=\langle[a,b]\rangle^G$. A direct computation shows that $K$ has
index $16$ in $G$, using the relations
$a^2=b^2=c^2=d^2=bcd=(ad)^4=1$. The computation
$\psi([[a,b],d])=\pair<1,[a,b]>$ shows that $\psi(K)$ contains
$K\times K$.

In the corresponding branch structure, one has $B=C_2\times D_8$.

Another direct computation shows that $[K,K]$ contains
$\psi^{-3}(K^{2^3})$, so that, by Remark~\ref{rem:[K,K]}, the
representations of $G$ and $G(B)$ are in bijection.

\subsection{The Gupta-Sidki groups}\label{ss:branch:gs}
The maximal branching subgroups of the Gupta-Sidki groups
(see~\S\ref{ss:gs}) are well-known; we recall them briefly.

In the Gupta-Sidki group $G_p$, consider the subgroup $K=[G_p,G_p]$. A
direct computation shows that $K$ has index $p^2$ in $G_p$. If
$p\ge5$, then the computation $\psi([t,t^a])=\pair<[a,t],1,\dots,1>$
shows that $\psi(K)$ contains $K^p$. For $p=3$, the computation is
slightly different: $\psi([tt^a,t^at^{a^2}])=\pair<[t^{-1},a^{-1}],1,1>$.

In the corresponding branch structure, one has $B=C_p\times C_p$.

Another direct computation shows that $[K,K]$ contains
$\psi^{-2}(K^{p^2})$, so that, by Remark~\ref{rem:[K,K]}, the
representations of $G$ and $G(B)$ are in bijection.

\subsection{Non-faithful actions}\label{ss:branch:nonfaithful}
If $G$ is a self-similar branched group, but is not faithful, it may
still be possible to construct a branch structure for it. Consider the
example of~\S\ref{ss:nonfaithful}: it is a group of the form
$G=H\wr_{\mathcal O} W$, for an abelian group $H$, a self-similar
branched group $W$ and an orbit $\mathcal O$ of $W$ on the boundary of
the tree $X^*$.

Let $(B,\phi)$ be a branch structure for $W$, with $B\wr_X Q\supseteq
B_+\overset\phi\twoheadrightarrow B$. Set $B'=H\times B$ and
$B'_+=H^X\times B_+\subseteq B'\wr_XQ$, and define $\phi'\colon B'_+\to B'$ by
\begin{equation}\label{eq:bsnonfaithful}
  \phi'\big((h_x)_{x\in X},b\big)=\Big(\prod_{x\in X}h_x,\phi(b)\Big).
\end{equation}
Then $(B',\phi')$ is a branch structure for $G$.

\section{Proof of Theorem~\ref{thm:main}}\label{ss:proof}
The criterion ``$G$ is rigid if and only if $K/[K,K]$ is finite'' is
Corollary~\ref{cor:rigid}.

\subsection{Abscissa of convergence}
The next statement of the Theorem asserts that the abscissa of
convergence of $\zeta_G$ is finite and positive. The proof follows very
closely that in~\cite{bartholdi-h:zetawreath}, so I only describe its
main steps.

\begin{prop}[See~\cite{bartholdi-h:zetawreath}*{Proposition~13}]
  The abscissa of convergence of $\zeta_G$ is positive.
\end{prop}
\begin{proof}
  We let $r_n$ denote the number of irreducible degree-$n$
  representations of $K$.  As a first step, there are infinitely many
  irreducible representations of $K$, so that, for every $B\in\N$,
  there exists $n$ such that $\sum_{j\le n}r_j\ge B$.

  For every integer $\ell$, there are then at least $B^{d^\ell}$
  representations of $K^{X^\ell}$ of degree at most $n^{d^\ell}$.

  Induce and extend these representations to $G$, and
  apply~\eqref{eq:index}: the index of $\psi^{-\ell}(K^{X^\ell})$ in
  $G$ is $[K:\psi^{-1}(K^X)]^{(d^\ell-1)/(d-1)}[G:K]\le k^{d^\ell}$
  for some constant $k$, so there are at least $(B/k)^{d^\ell}$
  irreducible representations of $G$ of degree at most
  $(nk)^{d^\ell}$. Choosing any $B>k$ gives the desired inequality
  $\sigma_0\ge\log(B/k)/\log(nk)$.
\end{proof}

\begin{prop}[See~\cite{bartholdi-h:zetawreath}*{Proposition~12}]
  The abscissa of convergence of $\zeta_G$ is finite.
\end{prop}
\begin{proof}
  Since the proof follows
  closely~\cite{bartholdi-h:zetawreath}*{Proposition~12}, let me only
  sketch the proof. Furthermore, the finiteness of the abscissa of
  convergence also implicitly follows from the functional equation.

  Let $r_n$ denote the number of irreducible, $n$-dimensional complex
  representations of $K$. We claim that there exist constants
  $A\in\N$ and $t>1$ such that
  \begin{equation}\label{eq:finite:ansatz}
    r_n\le \overline{r_n}:=A (n/\sigma_0(n))^t,
  \end{equation}
  with $\sigma_0(n)$ denoting the number of divisors of $n$.

  Up to replacing $X$ by $X^\ell$ for some $\ell\in\N$, we may assume
  that $K$ acts non-trivially on $X$. Indeed $G$ acts transitively on
  $X^\ell$, so since $K$ has finite index in $G$ it acts with
  boundedly many orbits.

  From Proposition~\ref{prop:[K,K]<kernel}, all representations of
  $K$ are representations of $K/\psi^{-\ell}([K,K]^{X^\ell})$ for some
  $\ell\in\N$. Let us denote by $r_{n,\ell}$ the number of those
  representations of $K$ that factor through
  $K/\psi^{-\ell}([K,K]^{X^\ell})$. We have $r_n=\sup_\ell
  r_{n,\ell}$, and $r_{n,0}=0$ for all $n\ge2$ while
  $r_{n,0}=[K:[K,K]]$. We prove by induction on $\ell$
  that~\eqref{eq:finite:ansatz} holds for all $\ell\in\N$.

  To compute $r_{n,\ell+1}$ in terms of $r_{m,\ell}$ for all $m|n$, we
  apply Theorem~\ref{thm:clifford}. We tensor $d$ representations of
  $K$ to obtain a representation of $K^X$, extend it to its inertia
  subgroup $I\le K$, and induce it to a representation of
  $K$. Therefore
  \[r_{n,\ell+1}\le\sum_{\substack{\psi^{-1}(K^X)\le I\le
      K\\n=n_1\cdots n_d e[K:I]}}r_{n_1}\cdots r_{n_d} N_e,
  \]
  with $N_e$ denoting the number of $e$-dimensional projective
  representations of $I/\psi^{-1}(K^X)$. We consider only $n\ge3$. The
  summands with $e[K:I]\ge2$ are easily controlled by a bound of the
  form $\overline{r_n}/2$, if $t$ is large enough (independently of
  $\ell$). Consider then summands with $K=I$ and $e=1$. The $d$-tuple
  of representations of $K$ we are inducing must then be constant on
  $K$-orbits, and since these orbits are non-trivial, there are
  repetitions in the $K$-tuple, diminishing the number of factors
  $r_{n_1}\cdots r_{n_d}$; so this term may again by bounded by
  $\overline{r_n}/2$, if $t$ is large enough (independently of
  $\ell$).
\end{proof}

\subsection{Functional equation}
We fix a branched group $G$, a branch structure $(B,\phi)$, and an
epimorphism $f\colon G\to B$. Let $K=\ker(f)$ denote the branching
subgroup. Up to replacing $K$ by $[K,K]$ if needed, we assume by
Proposition~\ref{prop:[K,K]<kernel} that every representation of $G$
factors through $G/\psi^{-\ell}(K^{X^\ell})$ for some $\ell\in\N$.

Let $\mathscr T$ denote a complete set of equivalence class
representatives of $B$-representation triples. Recall that $\mathscr
T$ is finite, being the disjoint union of the second cohomology groups
of all subgroups of $B$. For $\Theta$ a $B$-representation triple,
we denote by $[\Theta]$ its representative in $\mathscr T$.

Without loss of generality, we assume that whenever
$\Theta,\Theta'\in\mathscr T$ are representation triples such that
$\img(\Theta)$ and $\img(\Theta')$ are conjugate in $B$, say by $b\in
B$, then $\src(\Theta)=\src(\Theta')$ and $("\text{conjugation by
}b")\circ f(\Theta)=f(\Theta')$.

In order to compute the zeta function $\zeta_G(s)$, we introduce
Dirichlet series
\[\zeta_{G,\Theta}(s)=\sum_{\substack{\rho\in\widehat K\\ [(\rho,f)]=\Theta}}(\deg\rho)^{-s}.\]
Then, by Proposition~\ref{prop:cliffordtriple}, these series can then
be assembled into $\zeta_G$ as follows:
\begin{equation}\label{eq:fe:1}
  \zeta_G(s)=\sum_{\Theta=(\rho,f)\in\mathscr T}\zeta_{G,\Theta}(s)\zeta_{\Theta^\vee}(s)[B:\img(\Theta)]^{-1-s}.
\end{equation}

In fact, the functional equation we derive will have the following
form, equivalent to~\eqref{thm:main:fe0}, for polynomials $F_\Theta$
to be defined in~\eqref{eq:formula}:
\[\text{Equation }\eqref{eq:fe:1},\text{ and }
\zeta_{G,\Theta}(s)=F_\Theta(\{\zeta_{G,\Theta'}(s),\dots,\zeta_{G,\Theta'}(ds)\}_{\Theta'\in\mathscr T})\text{ for all }\Theta\in\mathscr T.
\]

For greater clarity, we consider
\[G_+=\psi(G)=\{\pair<g_x>q\in G\wr_X Q\mid\pair<f(g_x)>q\in B_+\},
\]
and produce a functional equation relating the zeta functions of $G$
and $G_+$. Since $G$ and $G_+$ are isomorphic (via $\psi$), we will be
done.

In a different language, we know from Remark~\ref{rem:zetaG(B)} that
the zeta functions of $G$ and of the profinite group $G(B)$ coincide,
and $G(B)=\varprojlim G_n$. The zeta function of $G$ is the
co\"efficient-wise limit of the zeta functions of the finite groups
$G_n$, and the functional equation~\eqref{eq:formula} may also be
interpreted as a functional equation between the zeta functions of
$G_n$ and $G_{n+1}$, with $G_n$ taking the role of $G$ and $G_{n+1}$
taking the role of $G_+$. Starting from $G_{-1}=B$, we obtain by
iteration and taking a limit the zeta function of $G(B)$.

For brevity of notation, we consider the free module $\Omega$ with
base $\mathscr T$ over the ring of Dirichlet series, and its element
\[\zeta_{G,\mathscr T}:=\sum_{\Theta\in\mathscr T}\zeta_{G,\Theta}\cdot\Theta.\]
An equation in $\Omega$ is a convenient way of writing $\#\mathscr T$
equations among zeta functions.

Theorem~\ref{thm:clifford} asserts that all representations of $K$ may
be obtained by running through all choices of $\rho_x$, extending
$\bigotimes_{x\in X}\rho_x$ to its inertia in $K\le G_+$, tensoring by
a projective representation, and inducing to $K$. We show that the
equivalence class of the obtained representation triple depends only
on the equivalence classes of the representation triples $(\rho_x,f)$
and the datum of which $\rho_x$ are equivalent:
\begin{prop}\label{prop:indep}
  Let $(\rho_x)_{x\in X}$ be a collection of irreducible
  representations of $K$, with associated representation triples
  $\Theta_x:=[(\rho_x,f_x)]\in\mathscr T$. Write
  \[\zeta(s)=\sum_{\sigma\in\widehat K\text{ extending
    }\bigotimes_x\rho_x}\dim(\sigma)^{-s}[(\sigma,f)]\in\Omega.\] Then
  $\zeta(s)\prod_{x\in X}\dim(\rho_x)^s$ depends only on the
  $\Theta_x$ and on the relation $\{(x,y):\rho_x\sim\rho_y\}\subseteq
  X^2$.
\end{prop}
\begin{proof}
  The inertia of $\bigotimes_x\rho_x$ in $G\wr_X Q$ has the form
  $(\prod_x G_{\rho_x})^{\tilde c}\rtimes P$, for some $\tilde c\in
  G^X$ and the subgroup $P\le Q$ consisting of all $q\in Q$ such that
  $\rho_x\sim\rho_{qx}$ for all $x\in X$. It is also the preimage by
  $f_1\colon G\wr_X Q\to B\wr_X Q$ of $H=(\prod_x\img(\Theta_x)\rtimes
  P)^c$ for some $c\in B^X$, and is therefore determined by the
  character triples $\Theta_x$ and the relation
  $\{(x,y):\rho_x\sim\rho_y\}$.

  Define then $I=\prod_x\src(\Theta_x)\rtimes P$, and $f_+\colon I\to
  B\wr_X Q$ by $f_+(\pair<g_x>p)=(\pair<f_x(g_x)>p)^c$. On
  $N:=\ker(f_+)=\prod_x\ker(f_x)$, define the representation
  $\rho_+=\bigotimes_x\rho_x$. Then $(\rho_+,f_+)$ is a $(B\wr_X
  Q)$-representation triple.

  Write $I_+=f_+^{-1}(B_+)$, and denote still by $f_+$ the restriction
  of $f_+$ to $I_+$. We obtain a $B_+$-representation triple
  $(\rho_+,f_+)$. Set $N_+=f_+^{-1}(\ker\phi)$. Let $\sigma$ run over
  all the extensions of $\rho_+$ to $N_+$, and note that $\sigma$'s
  inertia still lies in $I_+$, by Lemma~\ref{lem:extend}.

  Note that the representation $\rho_+$ was extended from $N$ to
  $N_+$; this extension degree is therefore expressible as
  $\dim(\sigma)/\dim(\rho_+)$.

  Consider then the induced representation triple $(\sigma,\phi\circ
  f_+)$. The induction degree is $[\ker\phi:\ker\phi\cap\img(f_+)]$.

  This recipe is based on Theorem~\ref{thm:clifford}, and follows
  Proposition~\ref{prop:cliffordtriple} producing all representations
  of $K$ out of representations of its normal subgroup
  $\psi^{-1}(K^X)$. The equivalence class of the representation triple
  $(\sigma,\phi\circ f_+)$ depends only on the classes $\Theta_x$ and
  on the choice of subgroup $H$, which in turn was dictated by the
  relation $\{(x,y):\rho_x\sim\rho_y\}$. Furthermore, the extension
  and induction degrees are determined by character triples as
  required.
\end{proof}

We are now ready to construct the functional equation expressing
$\zeta_{G_+,\mathscr T}$ in terms of $\zeta_{G,\mathscr T}$. We follow
Proposition~\ref{prop:indep} in writing $\zeta_{G_+,\mathscr T}$ as a
sum, over all $d$-tuples of character triples $(\Theta_x)_{x\in X}$,
of all representations of $K\le G_+$ whose restriction to $K^X$ is a
multiple of $\bigotimes_x\rho_x$ for representations $\rho_x$ of $K$
with $[(\rho_x,f)]=\Theta_x$ for all $x\in X$.

Once a family $(\Theta_x)_{x\in X}\in\mathscr T^X$ of
$B$-representation triples has been fixed, we sum over all possible
inertias of the corresponding tensor product of representations. Since
the inertia contains $K^X$, it suffices to consider its image in
$B\wr_X Q$. We are therefore led to enumerate all subgroups $H\le
B\wr_X Q$ satisfying the following two properties: $H\cap
B^X=\prod_x\img(\Theta_x)$; and, denoting by $P\le Q$ the image of $H$
in $Q$, the family $(\Theta_x)_{x\in X}$ is constant on
$P$-orbits. The first condition implies that abstractly
$H\cong\prod_x\img(\Theta_x)\rtimes P$, and in fact
$H=(\prod_x\img(\Theta_x)\rtimes P)^c$ for some $c\in B^X$.

We then consider all representations induced and extended from all
irreducible representations $\bigotimes_x\rho_x$ of $K^X$ such that
$[(\rho_x,f)]=\Theta_x$ and $\rho_x\sim\rho_y$ if and only if $x\in P
y$. For a $P$-orbit $Y$, we write $\Theta_Y:=\Theta_y$ for any $y\in Y$.

To conclude the enumeration, observe that the subgroups $H$ as above
form a lattice, under \emph{reverse} inclusion, so that the lattice's
maximal element is $\prod_x\img(\Theta_x)$. Let $\mu$ denote the
lattice's M\"obius function~\cite{rota:fct1}; so
$\sum_{H\ge H'\ge H''}\mu(H,H')=\delta_{H,H''}$.

It is convenient to replace the condition ``if and only if
$x\in P y$'' by ``if $x\in P y$'', and apply inclusion-exclusion on
the lattice of subgroups $H$.  Indeed, then, the contribution to
$\zeta_{G_+,\mathscr T}(s)$ is
$\zeta(s)\prod_x\dim\rho(\Theta_x)^s\prod_{\text{$P$-orbits
    $Y$}}\zeta_{G,\Theta_Y}(s)$ with $\zeta(s)$ as in
Proposition~\ref{prop:indep}.

\noindent We have arrived at the following formula expressing
$\zeta_{G_+,\mathscr T}$ in terms of $\zeta_{G,\mathscr T}$; recall
the notation $\Theta_+=(\rho_+,f_+)$ from the proof of
Proposition~\ref{prop:indep}:
\begin{multline}\label{eq:formula}
  \zeta_{G_+,\mathscr T}(s)=\sum_{(\Theta_x)\in\mathscr
    T^X}\sum_{\prod_x\img(\Theta_x)\le H\le B\wr_X Q}\sum_{\substack{\sigma\text{ induced from }(\rho_+,f_+)\\(\rho_+,f_+)\text{ extending }\bigotimes_x\rho(\Theta_+)}}\\
  [\ker\phi:\ker\phi\cap\img(f_+)]^{-1-s}\left(\frac{\dim(\sigma)}{\dim(\rho_+)}\right)^{-s}\times\\
  \times\sum_{\prod_x\img(\Theta_x)\le H'\le H}\mu(H,H') \prod_{Y\text{ orbit of $H'$ on $X$}}\zeta_{G,\Theta_Y}(\#Y s)\cdot[(\sigma,\phi\circ f_+)].
\end{multline}
This equation takes place in the module $\Omega$; by writing it in the
basis $\mathscr T$, one obtains equations
$\zeta_{G_+,\Theta}=F_\Theta(\{\zeta_{G,\Theta'}:\Theta'\in\mathscr
T\})$
for polynomials $F_\Theta$ with co\"efficients in
$\Q(2^{-s},\dots,P^{-s})$ for $P=\#B_+$, as required.

\subsection{Singularities}
We recall some arguments from~\cite{bartholdi-h:zetawreath}.  Let as
usual $\sigma_0$ denote the abscissa of convergence of $\zeta_G$; it
is the maximum of the absciss\ae\ of convergence of $\zeta_{G,\Theta}$
for all $\Theta\in\mathscr T$, since all $\zeta_{G,\Theta}$ are
positive-co\"efficient power series counting subsets of the
representations counted by $\zeta_G$, and combining to $\zeta_G$
by~\eqref{eq:fe:1}. Let $\mathscr H(k)$ denote the ring of holomorphic
functions in $\{\Re(s)>2^{-k}\sigma_0\}$. Observe then that
$\zeta_{G,\Theta}$ converges in $\mathscr H(0)$ for all
$\Theta\in\mathscr T$, and that
\[\mathscr H(0)\subset\mathscr H(1)\subset\cdots\subset\bigcup_{k\ge0}\mathscr T(k)=\{f\colon\{\Re(s)>0\}\to\C\}.
\]
Treating all variables $\zeta_{G,\Theta}(ks)$ with $k\ge2$ as co\"efficients,
the functional equation~\eqref{eq:formula} may be viewed as a
polynomial equation system in unknowns $\zeta_{G,\Theta}(s)$ and
co\"efficients in $\mathscr H(1)$. As such, it defines the
$\zeta_{G,\Theta}(s)$ as algebraic functions, in a finite extension of
$\mathscr H(1)$. More generally, let $\overline{\mathscr H(k)}$ denote
an algebraic closure of $\mathscr H(k)$; then, for every $k\ge1$, the
functional equation~\eqref{eq:formula} may be viewed as a polynomial
equation system in unknowns $\zeta_{G,\Theta}(s)$ and co\"efficients
in $\overline{\mathscr H(k)}$, hence describing
$\zeta_{G,\Theta}(s)\in\overline{\mathscr H(k)}$.

It remains to check that the leading co\"efficients in the functional
equation never vanish. To see that, consider a monomial
$S=\zeta_{G,\Theta_1}(s)\cdots\zeta_{G,\Theta_d}(s)$ in a term
of~\eqref{eq:formula}. It is associated with representations that
extend/induce from $\rho_1\otimes\cdots\otimes\rho_d$ whose inertia is
precisely $\prod_x G_{\rho_x}$, namely for which the group $P$ as
above is trivial. There is therefore no inclusion-exclusion, and the
co\"efficient of $S$ in the functional equation is the Dirichlet
polynomial counting representations of $(\prod_x G_{\rho_x}\cap
G_+)/K^X$ with given cocycle; in particular, this co\"efficient is
holomorphic in $\{\Re(s)>0\}$, and bounded away from $0$.

We have therefore shown that all the singularities in $\{\Re(s)>0\}$
of $\zeta_{G,\Theta}$ are algebraic; since an algebraic closure of the
ring of holomorphic functions may be taken as the ring of convergent
Puiseux series (see e.g.~\cite{eisenbud:ca2ag}*{Corollary~13.15}), we
have power series expansions in $s^{1/e}$ about all $s\in\C$ with
$\Re(s)>0$, and in particular in $\sigma_0$. The root order $e$ at
$\sigma_0$ is at least $2$, because $\sigma_0$ is a singularity of
$\zeta_G$, and it bounded by the degree $d$ of the functional
equation.

\subsection{Layered groups}
In the special case that $G\cong G\wr_X Q$, we recover Theorem~3
from~\cite{bartholdi-h:zetawreath} as follows: $B=1$, and there is a
single representation triple. The subgroups $H$ are then in bijection
with subgroups of $Q$. Theorems~1 and~3
in~\cite{bartholdi-h:zetawreath} were in fact written in terms of the
lattice of partitions of $X$; however, if two subgroups $Q,Q'$ induce
the same orbit partition on $X$, then these subgroups contribute many
times to~\eqref{eq:formula}, but that multiplicity is compensated by
the M\"obius function. Since the cohomology classes in question are
all trivial, the summation on all $\sigma$ may in fact be written via
the zeta function of $Q$.

\section{Implementation details}
The proof given in~\S\ref{ss:proof} is constructive enough that it can
be implemented easily in a computer algebra system such as
\textsc{Gap}~\cite{gap4:manual}. The code is freely available, and is
part of my package \textsc{Fr} designed to manipulate self-similar
groups. Some changes to the method given in~\S\ref{ss:proof} made the
computation more efficient.

\subsection{Representation triples}
Representation triples are objects consisting of a linear
representation and a homomorphism. Cohomology classes in
$H^2(G,\C^\times)$ are represented as $2$-cocycles, namely, as lists
of cyclotomic numbers indexed by $G\times G$.

A function computes the cocycle of a representation triple.

Another function converts a representation triple to an equivalent one in
which the marking is a Schur covering map.

More precisely, this function finds, given a representation triple
$\Theta$ and a list $\mathscr T$ of representation triples, the one
from the list that is equivalent to $\Theta$.

A function computes all the $B$-representation triples up to
equivalence. This is done by enumerating subgroups of $B$; computing
their Schur cover; and for each subgroup enumerating the characters of
the kernel of its Schur covering map.

A function computes all projective representations of a group with
given cocycle; the group and cocycle are respectively given to the
function as image and representation of a representation triple.

A function, given a projective representation $\rho$ of $G$ that is
equivalent to a linear one and an epimorphism $f\colon G\twoheadrightarrow
B$ such that the restriction of $\rho$ to $\ker(f)$ is linear,
computes all linear representations of $G$ that are equivalent to
$\rho$. These are in bijection with $H^1(B,\C^\times)$.

Finally, a function computes, given a linear representation $\rho$ of
$H$ and a group $G\ge H$, all irreducible representations of $G$ that
extend $\rho$.

\subsection{Constructing the functional equation}
The parameters stated in Theorem~\ref{thm:main} are $N=\#\mathscr T$
and $P=\#B$. In particular, the partial zeta functions $\zeta_i(s)$
are really $\zeta_{G,\Theta}(s)$, and the homogeneous polynomials
$F_i$ are really $F_\Theta$.

It is too costly to enumerate all subgroups $H$ as
in~\S\ref{ss:proof}. Rather, given the triples $(\Theta_x)_{x\in X}$,
we first compute all admissible partitions of $X$, namely those
$\mathscr P=(Y_1,\dots,Y_k)$ such that if $x,y$ are in the same part
then $\Theta_x=\Theta_y$. We denote by $Q_{\mathscr P}$ the
stabilizer of $\mathscr P$ in $Q$. We then define subsets
$\mathscr C_x$ of $B$, for every $x\in X$, as follows. For each part
$Y_i$, we choose a representative $x_i$; we let $\mathscr C_{x_i}$ be
a right transversal of the normalizer of $\img(\Theta_{x_i})$ in
$B$. For the other $x\in Y_i$, we let $\mathscr C_x$ be a right
transversal of $\img(\Theta_x)$ in $B$.

The corresponding subgroup $H$ of $B_+$ is
$(\prod_x\img(\Theta_x)\rtimes Q_{\mathscr P})^c$ for an arbitrary
choice of $c\in\prod_x\mathscr C_x$. We do not construct $H$
explicitly, but rather let $\mathscr I$, the ``possible inertias'', be
the list, for all choices of a partition $\mathscr P$ and
$c\in\prod_x\mathscr C_x$, of the homomorphism $f$ from
$I=\prod_x\src(\Theta_x)\rtimes Q_{\mathscr P}$ to $B\wr_X Q$ given
by $(\prod_x f(\Theta_x))\times id$ followed by conjugation by $c$.

We then construct a $\mathscr I\times\mathscr I$-matrix $\iota$, with
$\iota(f,f')=1$ if $\img(f)\le\img(f')$ and $\iota(f,f')=0$
otherwise. The M\"obius function of $\mathscr I$ is just the matrix
inverse of $\iota$.

Now, for every $f\in\mathscr I$, we compute the extensions $\sigma$ of
$\prod_x\rho(\Theta_x)$ to $f^{-1}(\ker\phi)$; and keep track of the
extension degree $e$ and the induction degree $i$, as well as the
representative of $\Theta'=(\sigma,\phi\circ f)$ in $\mathscr
T$. Summing over all $f'\in\mathscr I$ the expression
$\mu(f,f')e^{-s}i^{-1-s}$, we have just computed a term of
$F_{\Theta'}$. We repeat this for all tuples $(\Theta_x)_{x\in
  X}\in\mathscr T^X$.

\subsection{Using the functional equation}\label{ss:using}
To compute the co\"efficient of $n^{-s}$ in $\zeta_G$, it is
sufficient to work with Dirichlet series truncated at degree $n$. One
starts with the Dirichlet series $\zeta_{B,\mathscr T}$, which can
easily be computed because $B$ is a finite group, and iterates the
functional equation to obtain a fixed point. The iteration converges
because the polynomials $F_i-z_{i,1}$ are homogeneous of degree at
least two. This is how high-degree co\"efficients were computed.

On the other hand, to continue $\zeta_G$ analytically, one starts by
computing a large number of terms of $\zeta_{G,\mathscr T}$ as above,
up to, say, degree $n=10^{10}$, obtaining a Dirichlet polynomial. For
$s\in\C$ with sufficiently large real part, $\zeta_G(s)$ is well
approximated by the Dirichlet polynomials of $\zeta_{G,\mathscr T}$
and~\eqref{eq:fe:1}. For smaller values of $s$, one goes through the
functional equation~\eqref{eq:formula}, and replaces
$\zeta_{G,\Theta}(ks)$, whenever $k\ge2$, by its value using the
Dirichlet polynomial. What remains is a sequence of $\#\mathscr T$
polynomials with complex co\"efficients and in variables
$\{\zeta_{G,\Theta}(s)\}_{\Theta\in\mathscr T}$. Such a system can be
solved numerically, e.g.\ using PHC~\cite{verschelde:795} or the more
recent \textsc{Bertini}~\cite{bates+:bertini}. The system usually has
more than one solution, and one picks the relevant one; in particular,
for real $s$, one picks (following analytic continuation) the solution
in $\C^{\mathscr T}$ that is closest to the one computed for a
neighbouring $s$.

Finally, to obtain the abscissa of convergence, one restricts oneself
to real $s$; and finds, by repeated subdivision, the minimal $s$ such
that the solutions returned by numerically solving for
$\zeta_{G,\Theta}(s)$ remain all real. By the Landau-Phragm\'en
theorem mentioned in the Introduction, the abscissa of convergence is
a number $\sigma_0$ such that all $\zeta_{G,\Theta}(k\sigma_0)$ may be
accurately computed using the Dirichlet polynomial truncation, while
the polynomial system derived from the functional equation has a
multiple root at $\sigma_0$.


\bib{aleshin:burnside}{article}{
  author={Ale{\v {s}}in, Stanislav~V.},
  title={Finite automata and the Burnside problem for periodic groups},
  date={1972},
  journal={Mat. Zametki},
  volume={11},
  pages={319\ndash 328},
  review={\MR {46 \#265}},
}

\bib{avni:rational}{article}{
  author={Avni, Nir},
  title={Arithmetic groups have rational representation growth},
  journal={Ann. of Math. (2)},
  volume={174},
  date={2011},
  number={2},
  pages={1009\ndash 1056},
  issn={0003-486X},
  review={\MR {2831112}},
  doi={10.4007/annals.2011.174.2.6},
}

\bib{bartholdi-h:zetawreath}{article}{
  author={Bartholdi, Laurent},
  author={de la Harpe, Pierre},
  title={Representation zeta functions of wreath products with finite groups},
  journal={Groups Geom. Dyn.},
  volume={4},
  date={2010},
  number={2},
  pages={209\ndash 249},
  issn={1661-7207},
  review={\MR {2595090 (2011h:20010)}},
  doi={10.4171/GGD/81},
  eprint={arXiv:0809.0131},
}

\bib{bartholdi-s-z:cspbg}{article}{
  author={Bartholdi, Laurent},
  author={Siegenthaler, Olivier},
  author={Zalesskii, Pavel A.},
  title={The congruence subgroup problem for branch groups},
  volume={187},
  date={2012},
  journal={Isr. J. Math},
  pages={419\ndash 450},
  issn={0021-2172},
  review={\MR {2891709}},
  doi={10.1007/s11856-011-0086-5},
  eprint={arXiv:0902.3220},
}

\bib{bartholdi-g-s:bg}{article}{
  author={Bartholdi, Laurent},
  author={Grigorchuk, Rostislav I.},
  author={{\v {S}}uni{\'k}, Zoran},
  title={Branch groups},
  conference={ title={Handbook of algebra, Vol. 3}, },
  book={ publisher={North-Holland}, place={Amsterdam}, },
  date={2003},
  pages={989\ndash 1112},
  review={\MR {2035113 (2005f:20046)}},
  doi={10.1016/S1570-7954(03)80078-5},
  eprint={arXiv:math/0510294},
}

\bib{bates+:bertini}{book}{
  author={Bates, Daniel~J.},
  author={Hauenstein, Jonathan~D.},
  author={Sommese, Andrew~J.},
  author={Wampler, Charles~W.},
  title={Numerically solving polynomial systems with Bertini},
  series={Software, Environments, and Tools},
  volume={25},
  publisher={Society for Industrial and Applied Mathematics (SIAM), Philadelphia, PA},
  date={2013},
  pages={xx+352},
  isbn={978-1-611972-69-6},
  review={\MR {3155500}},
}

\bib{clifford:induced}{article}{
  author={Clifford, Alfred~H.},
  title={Representations induced in an invariant subgroup},
  journal={Ann. of Math. (2)},
  volume={38},
  date={1937},
  number={3},
  pages={533\ndash 550},
  issn={0003-486X},
  review={\MR {1503352}},
  doi={10.2307/1968599},
}

\bib{dress:oscillation}{article}{
  author={Dress, Fran{\c {c}}ois},
  title={Th\'eor\`emes d'oscillations et fonction de M\"obius},
  language={French},
  conference={ title={Seminar on number theory, 1983\ndash 1984}, address={Talence}, date={1983/1984}, },
  book={ publisher={Univ. Bordeaux I, Talence}, },
  date={1984},
  pages={Exp. No. 33, 33},
  review={\MR {784079 (86f:11070)}},
}

\bib{eisenbud:ca2ag}{book}{
  author={Eisenbud, David},
  title={Commutative algebra},
  series={Graduate Texts in Mathematics},
  volume={150},
  note={With a view toward algebraic geometry},
  publisher={Springer-Verlag, New York},
  date={1995},
  pages={xvi+785},
  isbn={0-387-94268-8},
  isbn={0-387-94269-6},
  review={\MR {1322960 (97a:13001)}},
  doi={10.1007/978-1-4612-5350-1},
}

\bib{gap4:manual}{manual}{
  title={GAP --- Groups, Algorithms, and Programming, Version 4.4.10},
  label={GAP08},
  author={The GAP~Group},
  date={2008},
  url={\texttt {http://www.gap-system.org}},
}

\bib{MR28:5082}{article}{
  author={Golod, Evgueni{\u \i }~S.},
  title={On nil-algebras and finitely approximable $p$-groups},
  date={1964},
  journal={Izv. Akad. Nauk SSSR Ser. Mat.},
  volume={28},
  pages={273\ndash 276},
  review={\MR {28 \#5082}},
}

\bib{grigorchuk:growth}{article}{
  author={Grigorchuk, Rostislav~I.},
  title={On the Milnor problem of group growth},
  date={1983},
  issn={0002-3264},
  journal={Dokl. Akad. Nauk SSSR},
  volume={271},
  number={1},
  pages={30\ndash 33},
  review={\MR {85g:20042}},
}

\bib{gupta-s:burnside}{article}{
  author={Gupta, Narain~D.},
  author={Sidki, Said~N.},
  title={On the Burnside problem for periodic groups},
  date={1983},
  journal={Math. Z.},
  volume={182},
  pages={385\ndash 388},
}

\bib{harpe:ggt}{book}{
  author={Harpe, Pierre~{de la}},
  title={Topics in geometric group theory},
  publisher={University of Chicago Press},
  address={Chicago, IL},
  date={2000},
  isbn={0-226-31719-6; 0-226-31721-8},
  review={\MR {2001i:20081}},
}

\bib{isaacs:ctfg}{book}{
  author={Isaacs, I. Martin},
  title={Character theory of finite groups},
  note={Corrected reprint of the 1976 original [Academic Press, New York; MR0460423]},
  publisher={AMS Chelsea Publishing, Providence, RI},
  date={2006},
  pages={xii+310},
  isbn={978-0-8218-4229-4},
  isbn={0-8218-4229-3},
  review={\MR {2270898}},
}

\bib{jaikin-zapirain:zetapadic}{article}{
  author={Jaikin-Zapirain, Andrei},
  title={Zeta function of representations of compact $p$-adic analytic groups},
  journal={J. Amer. Math. Soc.},
  volume={19},
  date={2006},
  number={1},
  pages={91\ndash 118 (electronic)},
  issn={0894-0347},
  review={\MR {2169043 (2006f:20029)}},
  doi={10.1090/S0894-0347-05-00501-1},
}

\bib{kargapolov-merzlyakov:otg}{book}{
  author={Kargapolov, Mihail~I.},
  author={Merzlyakov, Yuri{\u \i }~I.},
  title={{\cyreight Osnovy teorii grupp} (= basics of group theory)},
  edition={Third edition},
  publisher={``Nauka''},
  address={Moscow},
  date={1982},
  review={\MR {83k:20001}},
}

\bib{larsen-lubotzky:representationlinear}{article}{
  author={Larsen, Michael},
  author={Lubotzky, Alexander},
  title={Representation growth of linear groups},
  journal={J. Eur. Math. Soc. (JEMS)},
  volume={10},
  date={2008},
  number={2},
  pages={351\ndash 390},
  issn={1435-9855},
  review={\MR {2390327 (2009b:20080)}},
  doi={10.4171/JEMS/113},
}

\bib{lubotzky-martin:polynomialrepgrowth}{article}{
  author={Lubotzky, Alexander},
  author={Martin, Benjamin},
  title={Polynomial representation growth and the congruence subgroup problem},
  journal={Israel J. Math.},
  volume={144},
  date={2004},
  pages={293\ndash 316},
  issn={0021-2172},
  review={\MR {2121543 (2006b:20065)}},
  doi={10.1007/BF02916715},
}

\bib{nekrashevych:ssg}{book}{
  author={Nekrashevych, Volodymyr~V.},
  title={Self-similar groups},
  series={Mathematical Surveys and Monographs},
  volume={117},
  publisher={American Mathematical Society, Providence, RI},
  date={2005},
  pages={xii+231},
  isbn={0-8218-3831-8},
  review={\MR {2162164 (2006e:20047)}},
  doi={10.1090/surv/117},
}

\bib{passman-t:reps}{article}{
  author={Passman, Donald~S.},
  author={Temple, Will~V.},
  title={Representations of the Gupta-Sidki group},
  journal={Proc. Amer. Math. Soc.},
  volume={124},
  date={1996},
  number={5},
  pages={1403\ndash 1410},
  issn={0002-9939},
  review={\MR {1307556 (96g:20009)}},
  doi={10.1090/S0002-9939-96-03241-8},
}

\bib{rota:fct1}{article}{
  author={Rota, Gian-Carlo},
  title={On the foundations of combinatorial theory. I. Theory of M\"obius functions},
  journal={Z. Wahrscheinlichkeitstheorie und Verw. Gebiete},
  volume={2},
  date={1964},
  pages={340\ndash 368 (1964)},
  review={\MR {0174487 (30 \#4688)}},
}

\bib{schur:darstellung}{article}{
  author={Schur, Issai},
  title={\"Uber die Darstellung der endlichen Gruppen durch gebrochene lineare Substitutionen},
  language={German},
  journal={J. f\"ur Math.},
  volume={127},
  pages={20-50},
  year={1904},
  doi={10.1515/crll.1904.127.20},
}

\bib{schur:untersuchungen}{article}{
  author={Schur, Issai},
  title={Untersuchungen \"uber die Darstellung der endlichen Gruppen durch gebrochene lineare Substitutionen},
  language={German},
  journal={J. f\"ur Math.},
  volume={132},
  pages={85-137},
  year={1907},
  doi={10.1515/crll.1907.132.85},
}

\bib{siegenthaler:phd}{thesis}{
  author={Siegenthaler, Olivier},
  title={Discrete and profinite groups acting on regular rooted trees},
  type={Ph.D. Thesis},
  date={2009},
}

\bib{tenenbaum:apnt}{book}{
  author={Tenenbaum, G\'erald},
  title={Introduction to analytic and probabilistic number theory},
  series={Cambridge Studies in Advanced Mathematics},
  volume={46},
  note={Translated from the second French edition (1995) by C. B. Thomas},
  publisher={Cambridge University Press, Cambridge},
  date={1995},
  pages={xvi+448},
  isbn={0-521-41261-7},
  review={\MR {1342300 (97e:11005b)}},
}

\bib{verschelde:795}{article}{
  author={Verschelde, Jan},
  title={Algorithm 795: PHCpack: a general-purpose solver for polynomial systems by homotopy continuation},
  journal={ACM Trans. Math. Softw.},
  pages={251-276},
  year={1999},
}

\bib{witten:quantumgauge}{article}{
  author={Witten, Edward},
  title={On quantum gauge theories in two dimensions},
  journal={Comm. Math. Phys.},
  volume={141},
  date={1991},
  number={1},
  pages={153\ndash 209},
  issn={0010-3616},
  review={\MR {1133264 (93i:58164)}},
}



\begin{bibsection}
\begin{biblist}
\bibselect{math}
\end{biblist}
\end{bibsection}
\end{document}